\documentclass[11pt,a4paper]{article}
\title{\bf Dimension-Free Square Function Estimates for Dunkl Operators}
\author{Huaiqian Li\footnote{Email: {\color{blue}huaiqianlee@gmail.com}}
\quad Mingfeng Zhao\footnote{Email: {\color{blue}mingfeng.zhao@tju.edu.cn}}  \vspace{2mm}
\\
{\footnotesize Center for Applied Mathematics, Tianjin University, Tianjin 300072, P. R. China}
}

\date{}
\usepackage{amssymb,amsmath,amsfonts,amsthm,color,mathrsfs}
\usepackage{bbm}
\usepackage[marginal]{footmisc}

\usepackage{color} %%%
\def\R{\mathbb{R}}
\def\E{\mathbb{E}}

\def\d{\textup{d}}

\def\Hess{\textup{Hess}}

\def\supp{\textup{supp}}
\def\RCD{\textup{RCD}}

\def\<{\langle}
\def\>{\rangle}
\def\Proof.{\noindent{\bf Proof. }}

\def\newdot{{\kern.8pt\cdot\kern.8pt}}

\newtheorem{theorem}{Theorem}[section]
\newtheorem{lemma}[theorem]{Lemma}
\newtheorem{corollary}[theorem]{Corollary}
\newtheorem{proposition}[theorem]{Proposition}
\newtheorem{example}[theorem]{Example}

\theoremstyle{definition}\newtheorem{remark}[theorem]{Remark}

\begin{document}
\allowdisplaybreaks
\maketitle
\makeatletter % '@' is now a normal "letter" for TeX
\renewcommand\theequation{\thesection.\arabic{equation}}
\@addtoreset{equation}{section}
\makeatother % '@' is restored as a "non-letter" character for TeX

\begin{abstract}
Dunkl operators may be regarded as differential-difference operators parameterized by finite reflection groups and multiplicity functions. In this paper, the Littlewood--Paley square function for Dunkl heat flows in $\mathbb{R}^d$ is introduced by employing the full ``gradient'' induced by the corresponding carr\'{e} du champ operator and then the $L^p$ boundedness is studied for all $p\in(1,\infty)$. For $p\in(1,2]$, we successfully adapt Stein's heat flows approach to overcome the difficulty caused by the difference part of the Dunkl operator and establish the $L^p$ boundedness, while for $p\in[2,\infty)$, we restrict to a particular case when the corresponding Weyl group is isomorphic to $\mathbb{Z}_2^d$ and apply a probabilistic method to prove the $L^p$ boundedness. In the latter case, the curvature-dimension inequality for Dunkl operators in the sense of Bakry--Emery, which may be of independent interest, plays a crucial role. The results are dimension-free.
\end{abstract}



\section{Introduction and main results}\hskip\parindent
In this section, we first recall some basics on the Dunkl operator initially introduced by C.F. Dunkl in \cite{Dunkl1988,Dunkl1989}, and then we present the main results of this work. The Dunkl operator has been studied intensively since its introduction. For a general overview, refer to the survey papers \cite{Rosler2003,RosVoi2008,Anker2017}, as well as the books \cite{DunklXu2014,DX2015}.

Let $\R^d$ be endowed with the standard inner product $\langle\cdot,\cdot\rangle$ and the associated Euclidean norm $|\cdot|$. For $\alpha\in\R^d\setminus\{0\}$, let $H_\alpha$ be the hyperplane orthogonal to $\alpha$, i.e., $H_\alpha=\{x\in\R^d: \langle\alpha,x\rangle=0\}$, and denote $r_\alpha$ the reflection with respect to the hyperplane $H_\alpha$, which is a map from $\R^d$ to itself such that
$$r_\alpha x =x-2\frac{\langle \alpha,x\rangle}{|\alpha|^2}\alpha,\quad x\in\R^d.$$

A root system in $\R^d$ is a finite, nonempty subset of $\R^d\setminus\{0\}$, denoted by $\mathfrak{R}$, such that  for every $\alpha\in\mathfrak{R}$, $\mathfrak{R}\cap \alpha\R=\{\alpha,-\alpha\}$ and $r_\alpha(\mathfrak{R})=\mathfrak{R}$. Given such a root system $\mathfrak{R}$, denote $G$ the Weyl group (also celled reflection group or Coxeter group) generated by the reflections $\{r_\alpha:\alpha\in\mathfrak{R}\}$. It is well known that $G$ is a finite subgroup of the orthogonal group of $\R^d$.

The Weyl chambers associated to the root system $\mathfrak{R}$ are the connected components of $\{x\in\R^d: \langle\alpha,x\rangle\neq0\mbox{ for every }\alpha\in\mathfrak{R}\}=:W$. Given $y\in W$, we fix a positive subsystem $\mathfrak{R}_+ :=\{\alpha\in\mathfrak{R}: \langle\alpha,y\rangle>0\}$. Then, for every $\alpha\in \mathfrak{R}$, either $\alpha\in\mathfrak{R}_+$ or $-\alpha\in\mathfrak{R}_+$. In other words, $\mathfrak{R}$ can be written as the disjoint union of  subsystems $\mathfrak{R}_+$ and $-\mathfrak{R}_+$.

Let $\kappa_\cdot: \mathfrak{R}\rightarrow\mathbb{C}$ be a $G$-invariant function, i.e., $\kappa_{g\alpha}=\kappa_\alpha$ for every $g\in G$ and every $\alpha\in\mathfrak{R}$. We should mention that, due to the $G$-invariance of $\kappa$, the particular choice of $\mathfrak{R}_+$ makes no difference in the definition of Dunkl operators below. So we can fix any subsystem $\mathfrak{R}_+$ from now on.

Without loss of generality, we may normalize the root system such that $|\alpha|=\sqrt{2}$ for every $\alpha\in\mathfrak{R}$. For $\xi\in\R^d$, the Dunkl operator $D_\xi$ along $\xi$ associated to the root system $\mathfrak{R}$ and the function $\kappa$ is defined by
$$D_\xi f(x):=\partial_\xi f(x)+\sum_{\alpha\in\mathfrak{R}_+}\kappa_\alpha \langle\alpha,\xi\rangle \frac{f(x)-f(r_\alpha x)}{\langle\alpha,x\rangle},\quad f\in C^1(\R^d),\,x\in\R^d,$$
where $\partial_\xi$ denotes the directional derivative along $\xi$.

A remarkable property of Dunkl operators is the commutativity, i.e., for every $\xi,\eta\in\R^d$, $D_\xi\circ D_\eta=D_\eta\circ D_\xi$ (see \cite[Theorem 1.9]{Dunkl1989} or \cite[Theorem 6.4.8]{DunklXu2014}). Let $\{e_l: l=1,\cdots,d\}$ be the standard orthonormal basis of $\R^d$. For convenience, we write $D_l=D_{e_l}$, $l=1,\cdots,d$. We denote $\nabla_\kappa=(D_1,\cdots,D_d)$ the Dunkl gradient and $\Delta_\kappa=\sum_{l=1}^d D_l^2$ the Dunkl Laplacian. By \cite{Dunkl1989} or \cite{DunklXu2014}, we can express specifically that, for every $f\in C^2(\R^d)$,
\begin{eqnarray}\label{Delta}
	\Delta_\kappa f(x)=\Delta f(x)+2\sum_{\alpha\in\mathfrak{R}_+}\kappa_\alpha\Big(\frac{\langle\alpha,\nabla f(x)\rangle}{\langle\alpha,x\rangle} - \frac{f(x)-f(r_\alpha x)}{\langle\alpha,x\rangle^2}\Big),\quad x\in\R^d,
\end{eqnarray}
which indicates that $\Delta_\kappa$ is a differential-difference operator. It is easy to see that when $\kappa=0$, $D_\xi=\partial_\xi$, and $\nabla_0=\nabla$ and $\Delta_0=\Delta$ are the classical gradient operator and the Laplacian, respectively.

A typical example is the rank-one case.
\begin{example}\label{rank-one}
 Let $d=1$.  Then the only choice of the root system is $\mathfrak{R}=\{-\sqrt{2},\sqrt{2}\}$, the corresponding Weyl group is $G=\{{\rm id}, r\}$ with ${\rm id}(x)=x$ and $r(x)=-x$ for every $x\in\R$, and the multiplicity function is the constant function. Given a constant $\kappa\in\mathbb{C}$, the Dunkl operator $D=D_1$ is expressed as
\begin{equation}\label{1D-Dunkl}
Df(x)=f'(x)+\kappa\frac{f(x)-f(-x)}{x},\quad f\in C^1(\R),\, x\in\R,
\end{equation}
and the Dunkl Laplacian is written as
$$\Delta_\kappa f(x)=D^2f(x)=f''(x)+\frac{\kappa}{x^2}\big[f(-x)-f(x)+2xf'(x)\big],\quad f\in C^2(\R),\,x\in\R.$$
\end{example}

Another interesting example is the radial Dunkl process. See also \cite{RosVoi1998} and the recent \cite{GLR2018,GS2020} for more details. Here and below, $\mathbbm{1}$ denotes the constant function equal to $1$.
\begin{example}\label{Dyson-BM}
Let $C=\{x\in\R^d: \langle\alpha,x\rangle>0
\mbox{ for every }\alpha\in\mathfrak{R}_+\}$, and let $\overline{C}$ be its closure. The radial Dunkl process is defined as the $\overline{C}$-valued Markov process with continuous path, whose infinitesimal generator is given by
$$\Delta_\kappa^W f(x)=\Delta f(x)+2\sum_{\alpha\in\mathfrak{R}_+}\kappa_\alpha \frac{\langle\alpha,\nabla f(x)\rangle}{\langle\alpha,x\rangle},\quad x\in\R^d,$$
where $f$ belongs to $C^2(\overline{C})$ satisfying the boundary condition $\langle\alpha,\nabla f(x)\rangle=0$ for every $x\in H_\alpha,\,\alpha\in\mathfrak{R}_+$. Note that when the Weyl group $G=\mathbb{S}^{d-1}$, the symmetric group, and $\kappa=\mathbbm{1}$, $\Delta_\kappa^W$ is connected with the infinitesimal generator of the $d$-dimensional Dyson Brownian motion. In particular, when $\kappa=0$, $\Delta_\kappa^W$ is just the infinitesimal generator of the $d$-dimensional Brownian motion with reflection.
\end{example}

From now on, we assume that $\kappa\geq0$ and fix it. The natural weight function associated to the Dunkl operator is
$$\prod_{\alpha\in\mathfrak{R}_+}|\langle\alpha,x\rangle|^{2\kappa_\alpha}=:w_\kappa(x),\quad x\in\R^d,$$
which is a homogeneous function of degree $2\gamma$ with $\gamma=\sum_{\alpha\in\mathfrak{R}_+}\kappa_\alpha$ and also $G$-invariant. Obviously, $w_0=\mathbbm{1}$. For convenience, set $\d\mu_\kappa(x)=w_\kappa(x)\d x$. For each $p\in[1,\infty]$, we use $L^p(\mu_\kappa)$ and $\|\cdot\|_{L^p(\mu_\kappa)}$ to denote the classical $L^p$ space $L^p(\R^d,\mu_\kappa)$ and its $L^p$ norm, respectively.

The Dunkl Laplacian $\Delta_\kappa$ is essentially self-adjoint in $L^2(\mu_\kappa)$ (see e.g. \cite[Corollary 2.40]{RosVoi2008}). It generates the Dunkl heat flow $(H_\kappa(t))_{t\geq0}$ in $L^2(\mu_\kappa)$ as
$$H_\kappa(t) f(x)=e^{t\Delta_\kappa}f(x)=\int_{\R^d}h_\kappa(t,x,y)f(y)\,\d\mu_\kappa(y),\quad x\in\R^d,\,t>0,$$
and $H_\kappa(0)f=f$. Here $h_\kappa(t,x,y)$ is the Dunkl heat kernel, which is symmetric in $x$ and $y$, smooth in $(t,x,y)\in(0,\infty)\times\R^d\times\R^d$, positive, stochastically complete, i.e., 
\begin{equation}\label{stoch-complete}
\int_{\R^d} h_\kappa(t,x,y)\,\d\mu_\kappa(y)=1,\quad x\in\R^d,\,t>0,
 \end{equation} and satisfies the semigroup identity; see \cite[Section 4]{ADH2019} for more details on the Dunkl heat kernel and its estimates. It turns out that $(H_\kappa(t))_{t\geq0}$ is a strongly continuous contraction semigroup in $L^2(\mu_\kappa)$, and $H_\kappa(\cdot) f: (0,\infty)\rightarrow L^2(\mu_\kappa)$ is the unique continuously differentiable map, with values in the domain $\mathcal{D}(\Delta_\kappa)$, such that
\begin{eqnarray*}
\begin{cases}
\frac{\partial}{\partial t}H_\kappa(t)f=\Delta_\kappa H_\kappa(t) f,\quad&{\mbox{for }t\in (0,\infty)},\\
\lim_{t\rightarrow0^+}H_\kappa(t)f=f,\quad&{\mbox{in }L^2(\mu_\kappa)}.
\end{cases}
\end{eqnarray*}
Moreover, for $1\leq p<\infty$, $(H_\kappa(t))_{t\geq0}$ can be extended uniquely to a strongly continuous contraction semigroup in $L^p(\mu_\kappa)$, for which we keep the same notation. See \cite{RosVoi1998,Rosler2003,RosVoi2008} for more details. Furthermore, from \cite[Theorem 1 on page 67]{Stein70}, we see that $(H_\kappa(t))_{t\geq0}$ can be extended to an analytic semigroup in $L^p(\mu_\kappa)$ when $1<p<\infty$, and we also keep the notation the same. Indeed,  by \cite[Theorem 2.7]{Del2011}, $(H_\kappa(t))_{t\geq0}$ is a symmetric diffusion semigroup in the sense of \cite[page 65]{Stein70}.

As in the classical Laplacian case, we introduce the carr\'{e} du champ operator $\Gamma$ (see e.g. \cite{BE1985}): for $f,g\in C^2(\R^d)$,
$$\Gamma(f,g):=\frac{1}{2}\big[\Delta_\kappa(fg)-f\Delta_k g-g\Delta_\kappa f\big].$$
For convenience, set $\Gamma(f)=\Gamma(f,f)$. By a straightforward calculation, we get that, for every $f,g\in C^2(\R^d)$,
\begin{equation}\label{Gamma}
\Gamma(f,g)(x)=\langle\nabla f(x),\nabla g(x)\rangle + \sum_{\alpha\in\mathfrak{R}_+}\kappa_\alpha\frac{\big(f(x)-f(r_\alpha x)\big)\big(g(x)-g(r_\alpha x)\big)}{\langle\alpha,x\rangle^2},\quad x\in\R^d;
\end{equation}
see also \cite[Lemma 4.4]{GLR2018} and \cite[Lemma 3.1]{Velicu2020}.

For $f\in C_c^\infty(\R^d)$, define the Littlewood--Paley square function $g_\Gamma(f)$ by
\begin{eqnarray*}
g_\Gamma(f)(x)=\Big(\int_0^\infty \Gamma\big(H_\kappa(t)f\big)(x)\,\d t\Big)^{1/2},\quad x\in\R^d.
\end{eqnarray*}

The operator $g_\Gamma$, which is obviously nonlinear, is the major study object of the present work. Let $p\in (1,\infty)$. We say that the operator $g_\Gamma$ is bounded in $L^p(\mu_\kappa)$ if, there exists a positive constant $C_p$ such that
$$\|g_\Gamma(f)\|_{L^p(\mu_\kappa)}\leq C_p\|f\|_{L^p(\mu_\kappa)},\quad f\in L^p(\mu_\kappa).$$

With these preparations in hand,  we can present the main result of the paper in the following theorem. Let $|\mathfrak{R}_+|$ be the order of $\mathfrak{R}_+$.
\begin{theorem}\label{main-thm}
For $p\in(1,2]$, the operator $g_\Gamma$ is bounded in $L^p(\mu_\kappa)$. For $p\in [2,\infty)$, if, in addition, the Weyl group $G$ is isomorphic to $\mathbb{Z}_2^d$, then the operator $g_\Gamma$ is bounded in $L^p(\mu_\kappa)$. Moreover, the constant depends only on $p$ and $|\mathfrak{R}_+|$ in the former case, while,  the constant depends only on $p$ in the latter case.
\end{theorem}

Some remarks are in order.
\begin{remark}\label{remark}
\begin{itemize}
\item[(i)] Let $f\in C_c^\infty(\R^d)$. Define the square function $g_{\nabla_\kappa}(f)$ as
\begin{equation*}\label{g-1}
g_{\nabla_\kappa}(f)(x)=\Big(\int_0^\infty |\nabla_\kappa H_\kappa(t)f|^2(x)\,\d t\Big)^{1/2},\quad x\in\R^d.
\end{equation*}
Then $g_{\nabla_\kappa}$ is bounded in $L^p(\mu_\kappa)$ provided that $g_\Gamma$ is bounded in $L^p(\mu_\kappa)$, since $g_{\nabla_\kappa}(f)$ is controlled pointwise by $g_\Gamma(f)$ due to the fact that\footnote{The authors are grateful to Professor Yacoub Chokri for pointing out the constant $1+2\gamma$ in \eqref{grad-Gamma}, which is better than $2\max\{1,2\gamma\}$ appeared in a former version of this manuscript.}
\begin{eqnarray}\label{grad-Gamma}
|\nabla_\kappa f|^2\leq (1+2\gamma)\Gamma(f).
\end{eqnarray}
%%%%%where $C_\gamma:=1+2\gamma$
 Indeed, for every $\alpha\in\mathfrak{R}_+$ and $x\in\R^d$, letting $(\delta_\alpha f)(x)=[f(x)-f(r_\alpha x)]/\langle \alpha,x\rangle$, we have
\begin{eqnarray*}
|\nabla_\kappa f|^2(x)
&=&\sum_{j=1}^d\Big|\partial_j f(x)+\sum_{\alpha\in \mathfrak{R}_+}\kappa_\alpha(\delta_\alpha f)(x)\alpha_j   \Big|^2\\
&=&\sum_{j=1}^d\Big[\big(\partial_j f(x)\big)^2+2\sum_{\alpha\in \mathfrak{R}_+}\kappa_\alpha \partial_j f(x) (\delta_\alpha f)(x)\alpha_j  +\Big(\sum_{\alpha\in \mathfrak{R}_+}\kappa_\alpha \alpha_j (\delta_\alpha f)(x)\Big)^2   \Big]\\
&\leq&|\nabla f|^2+\sum_{j=1}^d\sum_{\alpha\in \mathfrak{R}_+}\Big(2\kappa_\alpha\big(\partial_j f(x)\big)^2+\frac{\kappa_\alpha}{2}(\delta_\alpha f)^2(x)|\alpha_j|^2  \Big)\\
&&+2\Big(\sum_{\alpha\in \mathfrak{R}_+}\kappa_\alpha\Big)\Big(\sum_{\alpha\in \mathfrak{R}_+}\kappa_\alpha(\delta_\alpha f)^2(x)\Big)\\
&=&(1+2\gamma)\Gamma(f)(x).
\end{eqnarray*}

Also, if we define the square function $g_\nabla(f)$ as
\begin{equation*}\label{g-2}
g_\nabla(f)(x)=\Big(\int_0^\infty |\nabla H_\kappa(t)f|^2(x)\,\d t\Big)^{1/2},\quad x\in\R^d,
\end{equation*}
then $g_\nabla$ is a bounded operator in $L^p(\mu_\kappa)$ provided that $g_\Gamma$ is bounded in $L^p(\mu_\kappa)$, since $|\nabla f|^2(x)\leq\Gamma(f)(x)$ for every $x\in\R^d$, which obviously follows from \eqref{Gamma}.

\item[(ii)] We do not consider square functions defined by the Dunkl Poisson flow (see e.g. \cite{RosVoi1998}) $(P_\kappa(t))_{t\geq0}$, where $P_\kappa(t) :=e^{-t\sqrt{-\Delta_\kappa}}$, $t\geq0$.  The reason is that, if for every $f\in C_c^\infty(\R^d)$, define
    $$G_\Gamma(f)=\Big(\int_0^\infty t\Gamma\big(P_\kappa(t)f\big)(x)\,\d t\Big)^{1/2},\quad x\in\R^d,$$
    then $G_\Gamma(f)(x)\leq g_\Gamma(f)(x)$ for every $x\in\R^d$, which deduces in particular that the $L^p$ boundedness of $g_\Gamma$ implies the $L^p$ boundedness of $G_\Gamma$. Indeed, by applying the formula
$$e^{-t \sqrt{-\Delta_\kappa}}=\frac{1}{\sqrt{\pi}}\int_0^\infty  e^{\frac{t^2}{4u}\Delta_k} e^{-u} u^{-\frac{1}{2}}\,du,\quad t\geq0,$$
we derive that
\begin{align*}G_\Gamma(f)^2(x)
&=\int_0^\infty t\Gamma\Big(\int_0^\infty  e^{\frac{t^2}{4u}\Delta_\kappa}f(\cdot)e^{-u}u^{-\frac{1}{2}}\,\frac{\d u}{\sqrt{\pi}}\Big)(x)\,\d t\\
&\le \frac{1}{\sqrt{\pi}}\int_0^\infty t \int_0^\infty \Gamma\big(e^{\frac{t^2}{4u}\Delta_\kappa}f\big)(x)\,e^{-u}u^{-1/2}\,\d u \, \d t\\
&=\frac{1}{\sqrt{\pi}}\int_0^\infty \Big(\int_0^\infty  t\Gamma\big(e^{\frac{t^2}{4u}\Delta_\kappa}f\big)(x)\,\d t\Big)e^{-u}u^{-1/2}\,\d u\\
&=\frac{2}{\sqrt{\pi}} \Big(\int_0^\infty e^{-u}u^{1/2}\,\d u\Big)\Big(\int_0^\infty  \Gamma\big( e^{s\Delta_\kappa}f\big)(x)\,\d s\Big)\\
&=g_\Gamma(f)^2(x),
 \end{align*}
 where  we also used Jensen's inequality (see \eqref{Gamma} for the explicit expression of the bilinear operator $\Gamma$),  Fubini's theorem, the change-of-variables formula and the facts
$$\int_0^\infty e^{-u} u^{-1/2}\,du=\sqrt{\pi},\quad \int_0^\infty e^{-u} u^{1/2}\,du=\frac{\sqrt{\pi}}{2};$$
hence $G_\Gamma(f) \leq g_\Gamma(f) $.

Moreover, we can define $G_{\nabla_\kappa}$ and $G_{\nabla}$  similar as $g_{\nabla_\kappa}$ and $g_{\nabla}$, by employing the Dunkl Poisson flow instead of the Dunkl heat flow. Then, similar as in (i), the $L^p$ boundedness of $g_\Gamma$ implies the $L^p$ boundedness of both $G_{\nabla_\kappa}$ and $G_{\nabla}$.

\item[(iii)] In this work, we do not consider square functions defined by time derivatives such as
$$g_j(f):=\Big(\int_{\R^d}\Big|\frac{\partial^j}{\partial t^j}H_\kappa(t)f\Big|^2t^{2j-1}\,\d t\Big)^{1/2},\quad f\in C_c^\infty(\R^d),\,j=1,2,\cdots.$$
The reason is that, by the general result \cite[Corollary 1 on page 120]{Stein70}, it is well known that, for each $j=1,2,\cdots$, $g_j$ is bounded in $L^p(\mu_\kappa)$ for all $1<p<\infty$ with the constant independent of $d$, since $(H_\kappa(t))_{t>0}$ is a symmetric diffusion semigroup in the sense of E.M. Stein mentioned above. Clearly, the same result holds if $H_\kappa(t)$ is replaced by $P_\kappa(t)$.
\end{itemize}
\end{remark}

As a corollary of Theorem \ref{main-thm}, we summarize the results in Remark \ref{remark} (i)\&(ii) in the following corollary.
\begin{corollary}
For $p\in(1,2]$, $g_{\nabla_\kappa}$, $g_{\nabla}$, $G_{\Gamma}$, $G_{\nabla_\kappa}$ and $G_{\nabla}$ are bounded in $L^p(\mu_\kappa)$. For $p\in[2,\infty)$, if $G$ is isometric to $\mathbb{Z}_2^d$, then $g_{\nabla_\kappa}$, $g_{\nabla}$, $G_{\Gamma}$, $G_{\nabla_\kappa}$ and $G_{\nabla}$ are bounded in $L^p(\mu_\kappa)$. In each case, the constant is dimension-free.
\end{corollary}

It is well known that square functions are important in harmonic analysis and probability theory. For the classic theory of Littlewood--Paley square functions and its applications in multiplier theory, Sobolev spaces and Hardy spaces, refer to \cite{Stein1958,St1970,Stein70}. Despite extensive studies in various settings in literature, we concentrate on square function estimates in the Dunkl setting. For $p\in(1,2]$, the $L^p$ boundedness of the square function $G_\nabla$ for the Dunkl Poisson flow was obtained in \cite{Soltani2005} on $\R$ and in \cite{FSoltani2005} on $\R^d$, respectively. By establishing Banach space valued singular integral theory, for all $p\in(1,\infty)$, the $L^p$ boundedness of $G_{\nabla_\kappa}$  for the Dunkl poisson flow on $\R^d$ was obtained in \cite{AmSi2012}. Recently, for all $p\in(1,\infty)$, the $L^p$ boundedness of $G_\Gamma$ for the Dunkl poisson flow on $\R$ was obtained in \cite{LZL2017}, where the approach is based on a deep result from the theory of Hilbert space valued singular integrals. After submission of the present manuscript, in \cite{DH2020}, the authors considered square functions in their full generality and proved the $L^p$ boundedness for all $p\in(1,\infty)$ by using vector valued Cader\'{o}n--Zygmund theory. In contrast to the $L^p$ boundedness, in a recent work \cite{Li2020}, the first named author proved weak $(1,1)$ boundedness of $g_\Gamma$ by employing estimates of the Dunkl heat kernel and its derivatives.  We should point out that all of the above mentioned results are dimension dependent.

Our approach to prove Theorem \ref{main-thm} when $p\in (1,2]$ is motivated by the recent paper \cite{LiWang16} which deals with $L^p$ boundedness for square functions in the setting of Dirichlet forms of pure jump type in metric measure spaces. However, for $p\in(1,2]$,  in general, it is not possible to show the $L^p$ boundedness of the corresponding $g_\Gamma$ for Dirichlet forms of pure jump type; see \cite[EXAMPLE 2]{BBL2016} for a counterexample constructed by the $\alpha$-stable process with $\alpha=1/2$. In contrast to that, the Dunkl setting provides an interesting example of non-local nature such that $g_\Gamma$ is $L^p$-bounded for all $p\in(1,2]$. For $p\in[2,\infty)$, in general, although the Dunkl operator can be regarded as a non-local operator, it seems that we are not able to prove the $L^p$ boundedness of $g_\Gamma$ by applying the methods in \cite{LiWang16,BBL2016}. Instead, we restrict to the setting when the Weyl group $G$ is isomorphic to $\mathbb{Z}_2^d$. In this particular case, we can deal with the Dunkl process as a diffusion process.

We should mention that, to prove our results, we do not employ directly estimates of Gauss type for the Dunkl heat kernel and estimates on its time and space derivatives (see e.g. \cite[Theorem 2.2]{DH2020}) and the doubling property. We emphasize that our results are dimension-free.

Potential applications of the results on Dunkl multipliers and weighted $L^p$ boundedness of the above square functions are interesting subjects to be considered in future studies.

The remainder of the paper is organized as follows. The next two sections contain proofs of Theorem \ref{main-thm}. Section 2 serves to prove the case when $p\in(1,2]$, and Section 3 deals with the case when $p\in[2,\infty)$ and the Weyl group $G$ is isomorphic to $\mathbb{Z}_2^d$, where the curvature-dimension condition (see Proposition \ref{cur-dim} below) is employed. In Section 4, details on the proof of Proposition \ref{cur-dim} are presented. For a function $f$, let $f^+:=\max\{f,0\}$ and $f^-:=(-f)^+$. %%%%%%%%%%
The size of constants appeared in proofs are trackable.

\section{$L^p$ boundedness for $p\in(1,2]$}\hskip\parindent
In this section, we establish $L^p$ boundedness for $g_\Gamma$  in $L^p(\mu_\kappa)$ for all $1<p\leq 2$. We should mention that the idea of proof below is motivated by \cite[Section 2]{LiWang16} (see also \cite{Nick08} for the graph case), which may be regarded as a development of Stein's method in \cite{Stein70} for non-local operators.

Let $p\in(1,2]$. We introduce the pseudo-gradient ${\rm G}_p$ as follows:
$${\rm G}_p(f):=\frac{1}{p}\big[f^{2-p}\Delta_\kappa(f^p)-pf\Delta_\kappa f\big],$$
for some suitable nonnegative function $f$ defined on $\R^d$.

The next lemma provides an explicit expression for ${\rm G}_p(f)$; see also \cite[Lemmas 4.3 and 4.4]{GLR2018} for potentially related results. Let $0^0:=1$.
\begin{lemma}\label{Gamma-p-formula}
For $p\in(1,2]$, $0\leq f\in C_c^\infty(\R^d)$ and $x\in\R^d$, we have
\begin{eqnarray*}{\rm G}_p(f)(x)&=&(p-1)|\nabla f|^2(x)+2(p-1)\sum_{\alpha\in\mathfrak{R}_+}\frac{\kappa_\alpha}{\langle\alpha,x\rangle^2}\mathbbm{1}_{\{y\in\R^d: f(y)\neq f(r_\alpha y)\}}(x)\times\\
&&\big[f(r_\alpha x)-f(x)\big]^2\int_0^1\frac{f^{2-p}(x)(1-s)}{\big[(1-s)f(x)+sf(r_\alpha x) \big]^{2-p}}\,\d s.
\end{eqnarray*}
\end{lemma}
\begin{proof}
By direct calculations, we get
\begin{eqnarray*}
p{\rm G}_p(f)(x)&=&[f^{2-p}(x)\Delta_\kappa(f^p)(x)-pf(x)\Delta_\kappa (f)(x)]\\
&=&p(p-1)|\nabla f|^2(x)+ 2\sum_{\alpha\in\mathfrak{R}_+}\frac{\kappa_\alpha}{\langle\alpha, x\rangle^2}f^{2-p}(x)\big[\big(f^p(r_\alpha x)-f^p(x)\big)\\
&&-pf^{p-1}(x)\big(f(r_\alpha x)-f(x)\big)\big].
\end{eqnarray*}

By Taylor's expression of the function $t\mapsto t^p$ at the point $s$, and then by the change-of-variables formula, we have
\begin{eqnarray*}
t^p-s^p-ps^{p-1}(t-s)&=&p(p-1)\int_s^t u^{p-2}(t-u)\,\d u\\
&=&p(p-1)(t-s)^2\int_0^1 \frac{1-v}{[(1-v)s + vt]^{2-p}}\,\d v,
\end{eqnarray*}
for $s,t\geq0$ with $s\neq t$. Thus, if $f(r_\alpha x)\neq f(x)$, then letting $s=f(x)$ and $t=f(r_\alpha x)$, we finish the proof.
\end{proof}

From Lemma \ref{Gamma-p-formula}, we derive the following result which implies that $\Gamma(f)$ and ${\rm G}_p(f)$ are comparable in some pointwise sense.
\begin{lemma}\label{Gamma-comparison}
For $p\in(1,2]$, $0\leq f\in C_c^\infty(\R^d)$, there holds that
\begin{eqnarray}\label{Gamma-comparison-1}
\Gamma(f)(x)\geq \frac{1}{p-1}{\rm G}_p(f)(x)\geq0,\quad x\in\R^d,
\end{eqnarray}
and
\begin{eqnarray}\label{Gamma-comparison-2}
\Gamma(f)(x)\leq \frac{1}{p-1} \Big[\sum_{\alpha\in\mathfrak{R}_+}{\rm G}_p(f)(r_\alpha x) + {\rm G}_p(f)(x)\Big].
\end{eqnarray}
\end{lemma}
\begin{proof} (1) We prove \eqref{Gamma-comparison-1}. If $f(x)>f(r_\alpha x)\geq0$, then $(1-s)f(x)+sf(r_\alpha x)\geq f(r_\alpha x)$, $s\in[0,1]$, and hence
\begin{eqnarray*}
\int_0^1\frac{1-s}{\big[(1-s)f(x)+sf(r_\alpha x)\big]^{2-p}}\,\d s\leq f(r_\alpha x)^{p-2}\int_0^1 (1-s)\,\d s=\frac{1}{2}f(r_\alpha x)^{p-2}.
\end{eqnarray*}
If $f(r_\alpha x)>f(x)\geq0$, then $(1-s)f(x)+sf(r_\alpha x)\geq f(x)$, $s\in[0,1]$, and hence
\begin{eqnarray*}
\int_0^1\frac{1-s}{\big[(1-s)f(x)+sf(r_\alpha x)\big]^{2-p}}\,\d s\leq f(x)^{p-2}\int_0^1 (1-s)\,\d s=\frac{1}{2}f(x)^{p-2}.
\end{eqnarray*}
Thus, together with Lemma \ref{Gamma-p-formula}, we derive that
\begin{eqnarray*}
{\rm G}_p(f)(x)&\leq&(p-1)|\nabla f|^2(x)+(p-1)\sum_{\alpha\in\mathfrak{R}_+}\frac{\kappa_\alpha}{\langle\alpha,x\rangle^2}f^{2-p}(x)\times\\
&&\big(f(x)-f(r_\alpha x)\big)^2\big(f(x)\vee f(r_\alpha x)\big)^{p-2}\\
&\leq&(p-1)\Gamma(f)(x).
\end{eqnarray*}
It is immediate to see that ${\rm G}_p(f)\geq0$ from Lemma \ref{Gamma-p-formula}. Thus, \eqref{Gamma-comparison-1} is proved.

(2) Now we prove \eqref{Gamma-comparison-2}. Let
\begin{eqnarray*}
{\rm I}&=&\sum_{\alpha\in\mathfrak{R}_+}\kappa_\alpha\frac{\big(f(x)-f(r_\alpha x)\big)^2}{\langle\alpha,x\rangle^2}\mathbbm{1}_{\{y\in\R^d: f(r_\alpha y)<f(y)\}}(x),\\
{\rm II}&=&\sum_{\alpha\in\mathfrak{R}_+}\kappa_\alpha\frac{\big(f(x)-f(r_\alpha x)\big)^2}{\langle\alpha,x\rangle^2}\mathbbm{1}_{\{y\in\R^d: f(r_\alpha y)>f(y)\}}(x).
\end{eqnarray*}

For $f(r_\alpha x)<f(x)$, $(1-s)f(x)+ sf(r_\alpha x)\leq f(x)$, $s\in[0,1]$. Then
$$\int_0^1\frac{f^{2-p}(x)(1-s)}{\big[(1-s)f(x)+sf(r_\alpha x) \big]^{2-p}}\,\d s\geq\int_0^1(1-s)\,\d s=\frac{1}{2}.$$
Hence, by Lemma \ref{Gamma-p-formula},
\begin{eqnarray}\label{lemma2.2-2}
{\rm I}&\leq&  2\sum_{\alpha\in\mathfrak{R}_+}\kappa_\alpha\frac{\big(f(x)-f(r_\alpha x)\big)^2}{\langle\alpha,x\rangle^2} \mathbbm{1}_{\{y\in\R^d: f(r_\alpha y)<f(y)\}}(x)\cr
&&\times \int_0^1\frac{f^{2-p}(x)(1-s)}{\big[(1-s)f(x)+sf(r_\alpha x) \big]^{2-p}}\,\d s\cr
&\leq&\frac{1}{p-1}\big[{\rm G}_p(f)(x)-(p-1)|\nabla f|^2(x)\big].
\end{eqnarray}

For $f(r_\alpha x)>f(x)$, $(1-s)f(x)+ sf(r_\alpha x)\leq f(r_\alpha x)$, $s\in[0,1]$. Then
$$\int_0^1\frac{f^{2-p}(r_\alpha x)(1-s)}{\big[(1-s)f(r_\alpha x)+sf(x) \big]^{2-p}}\,\d s\geq\int_0^1(1-s)\,\d s=\frac{1}{2}.$$
Hence, since $\langle\alpha,r_\alpha x\rangle=-\langle\alpha,x\rangle$ for every $\alpha\in\mathfrak{R}_+$, by Lemma \ref{Gamma-p-formula} again, we have
\begin{eqnarray}\label{lemma2.2-3}
{\rm II}&\leq&2\sum_{\alpha\in\mathfrak{R}_+}\kappa_\alpha\frac{\big(f(r_\alpha x)-f(x)\big)^2}{\langle\alpha,r_\alpha x\rangle^2}\mathbbm{1}_{\{y\in\R^d: f(r_\alpha y)>f(y)\}}(x)\cr
 &&\times\int_0^1\frac{f^{2-p}(r_\alpha x)(1-s)}{\big[(1-s)f(r_\alpha x)+sf(x) \big]^{2-p}}\,\d s\cr
&\leq&\frac{1}{p-1}\sum_{\alpha\in\mathfrak{R}_+}\big[{\rm G}_p(f)(r_\alpha x)-(p-1)|\nabla f|^2(r_\alpha x)\big].
\end{eqnarray}

Thus, combining \eqref{lemma2.2-2} and \eqref{lemma2.2-3}, we finally arrive at
\begin{eqnarray*}\Gamma(f)(x)&=&|\nabla f|^2(x)+{\rm I}+{\rm II}\\
&\leq&\frac{1}{p-1}\Big[\sum_{\alpha\in\mathfrak{R}_+}{\rm G}_p(f)(r_\alpha x) + {\rm G}_p(f)(x)\Big],
\end{eqnarray*}
which is \eqref{Gamma-comparison-2}.
\end{proof}

Recall that $(H_\kappa(t))_{t\geq0}$ is the Dunkl heat flow. For every $0\leq f\in C_c^\infty(\R^d)$, define the square function $g_p(f)$ by
\begin{eqnarray*}
g_p(f)(x)=\Big(\int_0^\infty{\rm G}_p\big(H_\kappa(t)f\big)(x)\,\d t\Big)^{1/2},\quad  x\in\R^d.
\end{eqnarray*}

The next result is on the $L^p$ boundedness of the operator $g_p$.
\begin{proposition}\label{Lp-bound-g_p}
Let $p\in(1,2]$. Then there exists a constant $c_p\in(0,\infty)$, depending only on $p$, such that for all $0\leq f\in C_c^\infty(\R^d)$,
\begin{eqnarray}\label{Lp-bound-g_p-1}
\|g_p(f)\|_{L^p(\mu_\kappa)}\leq c_p\|f\|_{L^p(\mu_\kappa)},
\end{eqnarray}
and, moreover,
\begin{eqnarray}\label{Lp-bound-g_p-2}
\Big\|\sqrt{{\rm G}_p\big(H_\kappa(t)f\big)}\Big\|_{L^p(\mu_\kappa)}\leq\frac{c_p}{\sqrt{t}}\|f\|_{L^p(\mu_\kappa)},\quad  t>0.
\end{eqnarray}
\end{proposition}
\begin{proof}
Let $p\in(1,2]$. Assume that $0\leq f\in C_c^\infty(\R^d)$ and $f$ is not identical to the zero function. Then $H_\kappa(t)f\in C^\infty(\R^d)\cap\mathcal{D}(\Delta_\kappa)\cap L^p(\mu_\kappa)$ and $H_\kappa(t)f>0$ for all $t>0$. %%%%, since the Dunkl heat kernel is smooth and positive everywhere.
For notational simplicity, we set
$$v_t(x)=v(t,x)=H_\kappa(t)f(x).$$
Then
\begin{eqnarray*}
pv_t^{p-2}{\rm G}_p(v_t)&=&\Delta_\kappa(v_t^p)-pv_t^{p-1}\Delta_\kappa v_t\\
&=&pv_t^{p-1}(\partial_t -\Delta_\kappa) v_t - pv_t^{p-1}\partial_tv_t + \Delta_\kappa(v_t^p)\\
&=& (\Delta_\kappa-\partial_t)v_t^p=: L_t,
\end{eqnarray*}
where we used the fact that $(\partial_t -\Delta_\kappa) v_t =0$ in the last equality. Then
\begin{eqnarray}\label{Lp-bound-g_p-proof-1}
{\rm G}_p(v_t)=\frac{1}{p}v_t^{2-p}L_t.
\end{eqnarray}

(1) We prove \eqref{Lp-bound-g_p-1} first.
\begin{eqnarray*}
g_p(f)^2(x)&=&\int_0^\infty {\rm G}_p(v_t)(x)\,\d t\\
&=&\frac{1}{p}\int_0^\infty v^{2-p}(t,x)(\Delta_\kappa-\partial_t)v^p(t,x)\,\d t\\
&\leq&\frac{1}{p}\big(\sup_{t>0}v^{2-p}(t,x)\big)L(x),
\end{eqnarray*}
where we have let $L(x)=\int_0^\infty (\Delta_\kappa-\partial_t)v^p(t,x)\,\d t$, and we see that $L(x)\geq0$ since ${\rm G}_p(v_t)(x)\geq0$ by Lemma \ref{Gamma-comparison}. Thus
\begin{eqnarray}\label{proof-prop-1}
\int_{\R^d}g_p(f)^p\,\d\mu_\kappa&\leq& \frac{1}{p}\int_{\R^d} \big(\sup_{t>0}v(t,x)\big)^{(2-p)p/2}L(x)^{p/2}\,\d\mu_\kappa(x)\cr
&\leq&\frac{1}{p}\Big(\int_{\R^d}\big(\sup_{t>0}v(t,x)\big)^p\,\d\mu_\kappa(x)\Big)^{(2-p)/2}   \Big(\int_{\R^d}L(x)\,\d\mu_\kappa(x)\Big)^{p/2}\cr
&\leq& C_p\|f\|_{L^p(\mu_\kappa)}^{(2-p)p/2}\Big(\int_{\R^d}L(x)\,\d\mu_\kappa(x)\Big)^{p/2},
\end{eqnarray}
where we used the fact that $\|\sup_{t>0}v(t,x)\|_{L^p(\mu_\kappa)}\leq C_p\|f\|_{L^p(\mu_\kappa)}$ for some positive constant $C_p$ depending only on $p$ (see e.g. \cite[page 73]{Stein70}).

For every $t>0$, we claim that
\begin{eqnarray}\label{claim}
\int_{\R^d} \Delta_\kappa v_t^p\,\d\mu_\kappa\leq 0.
\end{eqnarray}
The proof of \eqref{claim} is motivated by \cite{LiWang16} in the case of pure jump Dirichlet forms. We may fix $t>0$ for the moment. By the analyticity of $H_\kappa(t)$, we have
\begin{eqnarray}\label{analyticity}
\|\Delta_\kappa v_t\|_{L^p(\mu_\kappa)}=\|\Delta_\kappa H_\kappa(t)f\|_{L^p(\mu_\kappa)}\leq \frac{a_p}{t}\|f\|_{L^p(\mu_\kappa)}.
\end{eqnarray}
for some positive constant $a_p$ depending only on $p$. Then, by H\"{o}lder's inequality and \eqref{analyticity},
\begin{eqnarray*}
\int_{\R^d}|\partial_t v_t^p|\,\d\mu_\kappa&=&p\int_{\R^d}|v_t|^{p-1}|\Delta_\kappa v_t|\,\d\mu_\kappa\\
&\leq&p\|v_t\|_{L^p(\mu_\kappa)}^{p-1}\|\Delta_\kappa v_t\|_{L^p(\mu_\kappa)}\leq \frac{pa_p}{t}\|f\|_{L^p(\mu_\kappa)}^p,
\end{eqnarray*}
which implies that $\partial_t v_t^p\in L^1(\mu_\kappa)$. By \eqref{Lp-bound-g_p-proof-1}, we get $L_t\geq0$.  Since
$$\Delta_\kappa v_t^p=L_t+ \partial_t v_t^p,$$
we have
$$\int_{\R^d}(\Delta_\kappa v_t^p)^-\,\d\mu_\kappa<\infty.$$
Let $\psi:\R\rightarrow[0,1]$ be an infinitely differentiable function such that $\supp(\psi)\subset(-\infty, 3)$,  $\psi(t)=1$ for every $t\leq1$,  and $|\psi'|+|\psi''|\leq C$ for some positive constant $C$. Given a sequence $(r_n)_{n\geq1}$ such that $r_n\rightarrow\infty$ as $n$ tends to $\infty$, define the sequence of cut-off functions $(\phi_n)_{n\geq1}$ by
$$\phi_n(x):=\psi\Big(\frac{|x|}{r_n}\Big),\quad x\in\R^d,\,n\geq1.$$
It is easy to see that $|\Delta \phi_n|\leq C_d/r_n^2$ for some positive constant $C_d$ depending only on $d$ and $C$. Note that $\phi_n$ is radial for each $n$. Then, by the generalized Fatou's Lemma (see \cite[Theorem 3.3.6 (2)]{YJA}) and the integration-by-parts formula (see e.g. \cite[Proposition 2.1]{Rosler2003}), we deduce that
\begin{eqnarray*}
\int_{\R^d}\Delta_\kappa v_t^p\,\d\mu_\kappa&\leq& \liminf_{n\rightarrow\infty}\int_{\R^d}\phi_n\Delta_\kappa v_t^p\,\d\mu_\kappa\\
&=&\liminf_{n\rightarrow\infty}\int_{\R^d}  v_t^p\Delta\phi_n\,\d\mu_\kappa\\
&\leq&\liminf_{n\rightarrow\infty}\frac{C_d}{r_n^2}\int_{\R^d} v_t^p \,\d\mu_\kappa=0.
\end{eqnarray*}
The claim is proved.

By \eqref{claim}, we have
\begin{eqnarray}\label{proof-prop-2}
\int_{\R^d}L(x)\,\d\mu_\kappa(x)&=&\int_0^\infty\int_{\R^d} \Delta_\kappa v_t^p(x)\,\d\mu_\kappa(x)\d t\cr
&&-   \int_{\R^d}\int_0^\infty \partial_t v_t^p(x)\,\d t\d\mu_\kappa(x)\cr
&\leq&\int_{\R^d}f^p(x)\,\d\mu_\kappa(x)=\|f\|_{L^p(\mu_\kappa)}^p.
\end{eqnarray}

Combining \eqref{proof-prop-1} and \eqref{proof-prop-2}, we obtain
$$\int_{\R^d}g_p(f)^p\,\d\mu_\kappa\leq C_p\|f\|_{L^p(\mu_\kappa)}^{(2-p)p/2}\|f\|_{L^p(\mu_\kappa)}^{p^2/2}=C_p\|f\|_{L^p(\mu_\kappa)}^p.$$
Thus, we complete the proof of \eqref{Lp-bound-g_p-1}.

(2) Now we prove \eqref{Lp-bound-g_p-2}. The argument is similar as above. By \eqref{claim} and H\"{o}lder's inequality,
\begin{eqnarray*}
\int_{\R^d}L_t\,\d\mu_\kappa&\leq& -\int_{\R^d}v_t^{p-1}\partial_t v_t\,\d\mu_\kappa\\
&\leq& \|v_t\|_{L^p(\mu_\kappa)}^{p-1}\|\partial_t v_t\|_{L^p(\mu_\kappa)}\\
&\leq&  \|f\|_{L^p(\mu_\kappa)}^{p-1} \frac{a_p}{t} \|f\|_{L^p(\mu_\kappa)}\\
&=&  \frac{a_p}{t} \|f\|_{L^p(\mu_\kappa)}^p,
\end{eqnarray*}
where we used the analyticity of $H_\kappa(t)$ (see \eqref{analyticity} above) in the last inequality. Hence, combining this together with \eqref{Lp-bound-g_p-proof-1} and applying H\"{o}lder's inequality, we obtain that
\begin{eqnarray*}
\int_{\R^d}{\rm G}_p(v_t)^{p/2}\,\d\mu_\kappa&=&\frac{1}{p}\int_{\R^d}v_t^{p(2-p)/2} L_t^{p/2}\,\d\mu_\kappa\\
&\leq&\frac{1}{p} \Big(\int_{\R^d}v_t^p\,\d\mu_\kappa\Big)^{(2-p)/2}   \Big(\int_{\R^d}L_t\,\d\mu_\kappa\Big)^{p/2}\\
&\leq& \frac{1}{p}\Big(\frac{a_p}{t}\Big)^{p/2}\|f\|_{L^p(\mu_\kappa)}^p.
\end{eqnarray*}
Thus, we complete the proof of \eqref{Lp-bound-g_p-2}.
\end{proof}

The main result in this section is presented in the next theorem. Recall that $|\mathfrak{R}_+|$ is the order of $\mathfrak{R}_+$.
\begin{theorem}\label{Gamma-LPS-small-p}
Let $p\in(1,2]$. Then the operator $g_\Gamma$ is bounded in $L^p(\mu_\kappa)$; more precisely, there exists a positive constant $C_p$, depending only on $p$ and $|\mathfrak{R}_+|$, such that for all $f\in L^p(\mu_\kappa)$,
\begin{eqnarray}\label{Gamma-LPS-small-p-1}
\|g_\Gamma(f)\|_{L^p(\mu_\kappa)}\leq C_p\|f\|_{L^p(\mu_\kappa)},
\end{eqnarray}
and, moreover,
\begin{eqnarray}\label{Gamma-LPS-small-p-2}
\Big\|\sqrt{\Gamma\big(H_\kappa(t)f\big)}\Big\|_{L^p(\mu_\kappa)}\leq\frac{C_p}{\sqrt{t}}\|f\|_{L^p(\mu_\kappa)},\quad t>0.
\end{eqnarray}
\end{theorem}
\begin{proof}
By standard approximation, it suffices to prove the case when $f\in C_c^\infty(\R^d)$. Assume $f\in C_c^\infty(\R^d)$. It is easy to see that
\begin{eqnarray*}\Gamma(H_\kappa(t)f)(x)&=&\Gamma\big(H_\kappa(t)(f^+-f^-)\big)(x)\\
&\leq& 2\big[\Gamma(H_\kappa(t)f^+)(x)+\Gamma(H_\kappa(t)f^-)(x)\big].
\end{eqnarray*}%%%%%%%%%%%%%%%%%%%
Then it is enough to assume $f\geq0$ in addition. By \eqref{Gamma-comparison-2} in Lemma \ref{Gamma-comparison}, we have, for every $\alpha\in\mathfrak{R}_+$,
\begin{eqnarray*}
g_\Gamma(f)^2(x)&=&\int_0^\infty \Gamma(H_\kappa(t) f)(x)\,\d t \\
&\leq&  c_p \int_0^\infty \Big[\sum_{\alpha\in\mathfrak{R}_+}{\rm G}_p(H_\kappa(t) f)(r_\alpha x) +  {\rm G}_p(H_\kappa(t) f)(x)\Big]\,\d t\\
&=&c_p\Big[\sum_{\alpha\in\mathfrak{R}_+}g_p(f)^2(r_\alpha x) + g_p(f)^2(x)\Big],
\end{eqnarray*}
where $c_p=1/(p-1)$ is from Lemma \ref{Gamma-comparison}. Applying \eqref{Lp-bound-g_p-1} in Proposition \ref{Lp-bound-g_p}, we deduce that
\begin{eqnarray*}
\int_{\R^d}g_\Gamma(f)^p\,\d\mu_\kappa&\leq&\tilde{c}_p\int_{\R^d}\Big[g_p(f)^p(x)+\sum_{\alpha\in\mathfrak{R}_+}g_p(f)^p(r_\alpha x)\Big]\,\d\mu_\kappa(x)\\
&=&\tilde{c}_p(1+|\mathfrak{R}_+|)\int_{\R^d}g_p(f)^p(x)\,\d\mu_\kappa(x)\\
&\leq& C_p\int_{\R^d}f^p\,\d\mu_\kappa,
\end{eqnarray*}
for some positive constant $C_p$ depending only on $p$ and $|\mathfrak{R}_+|$, where the equality is due to that $r_\alpha$ is a reflection and $\mu_\kappa$ is $G$-invariant. We complete the proof of \eqref{Gamma-LPS-small-p-1}.

Similarly, by \eqref{Gamma-comparison-2} again and \eqref{Lp-bound-g_p-2}, we complete the proof of \eqref{Gamma-LPS-small-p-2}.
\end{proof}

\section{$L^p$ boundedness for $p\in[2,\infty)$}\hskip\parindent
In this section, we prove the $L^p$ boundedness for the operator $g_\Gamma$ for all $p\in[2,\infty)$ in the particular case when the Weyl group $G$ is isomorphic to $\mathbb{Z}_2^d=\{0,1\}^d$. We employ the probabilistic approach which was initially introduced in \cite{BMH2003} for Brownian motions and was recently adapted successfully  to deal with diffusion processes on $\RCD(K,N)$ spaces (see \cite{Li2017} for the case when $K=0$ and $1\leq N<\infty$ and \cite{Li2019} for the case when $K\in\R$ and $N=\infty$, as well as for more details on $\RCD$ spaces).

The natural stochastic process generated by the Dunkl Laplacian is the so-called Dunkl process, which was studied earlier in \cite{Rosler1998,RosVoi1998,GaYor2005,GaYor2006} for instance. Let $X:=(X_t)_{t\geq0}$ be the Dunkl process with infinitesimal generator $\big(\Delta_\kappa,\mathcal{D}(\Delta_\kappa)\big)$ in $\R^d$. For each $\alpha\in\mathfrak{R}_+$, recall that $H_\alpha$ is the hyperplane orthogonal to $\alpha$. For every subset $I$ of $\mathfrak{R}_+$, let
$$U_I :=\{\alpha\in\mathfrak{R}_+: \langle\alpha,x\rangle=0,\,x\in \cap_{\alpha\in I}H_\alpha\}.$$
It is known that $X$ is a c\`{a}dl\`{a}g Markov process of jump type with jumping kernel (see \cite[Proposition 3.1]{GaYor2006})
\begin{eqnarray*}
J(x,dy)=\begin{cases}
\sum_{\alpha\in\mathfrak{R}_+}\frac{2\kappa_\alpha}{\langle\alpha,x\rangle^2}\delta_{r_\alpha x}(\d y),\quad&{x\in\R^d\setminus(\cup_{\alpha\in\mathfrak{R}_+}H_\alpha),}\\
\sum_{\alpha\in\mathfrak{R}_+\setminus U_I}\frac{2\kappa_\alpha}{\langle\alpha,x\rangle^2}\delta_{r_\alpha x}(\d y),\quad&{x\in \cap_{\alpha\in I}H_\alpha,}\\
0,\quad&{x=0,}
\end{cases}
\end{eqnarray*}
where $I$ is any subset of $\mathfrak{R}_+$, $\delta_z$ denotes the Dirac measure at the point $z\in\R^d$. Due to our purpose, we may assume that the process $X$ does not start from $0$ in what follows.

We should mention that although the Dunkl process $X$ is a jump process, the approaches developed mainly for pure jump L\'{e}vy processes in recent papers \cite{BBL2016,LiWang16} (see also \cite{LiWang2019} for the non-local Schr\"{o}dinger case) seem not applicable directly. However, the Dunkl heat flow in the special situation when the Weyl group $G$ is isomorphic to $\mathbb{Z}_2^d$ seems more well-behaved as the diffusion one. Due to this, we may apply the method used in \cite{Li2019}.

There are essentially no new ideas in the following arguments. The novelty here maybe is that we can calculate more explicitly in the present Dunkl setting than \cite[Section 3]{LiWang16} in the general setting of pure jump Dirichlet forms.

Now fix $f\in C^\infty_c(\R^d)$ and $T>0$. Let
$$\mathcal{N}_t:= H_\kappa(T-t)f(X_t)-H_\kappa(T)f(X_0),\quad t\in[0,T].$$
Let $(B_t)_{t\geq0}$ be the Brownian motion in $\R^d$ with infinitesimal generator $\Delta$, and denote by $(\mathcal{F}_t)_{t\geq0}$ the natural filtration of the process $X$.
\begin{lemma}\label{mart}
 $(\mathcal{N}_t,\mathcal{F}_t)_{t\in[0,T]}$ is a martingale starting from 0, and for any $t\in[0,T]$,
\begin{eqnarray}\label{quad-var-N}
\langle\mathcal{N}\rangle_t&=&2\int_0^t|\nabla H_\kappa(T-s)f|^2(X_s)\,\d s +\cr &&2\sum_{\alpha\in\mathfrak{R}_+}\int_0^t\kappa_\alpha\,\frac{\big(H_\kappa(T-s)f(X_{s-}) - H_\kappa(T-s)f(r_\alpha X_{s-})\big)^2}{\langle\alpha, X_{s-}\rangle^2}\,\d s,
\end{eqnarray}
where $\langle\mathcal{N}\rangle_t$ is the predictable quadratic variation of $\mathcal{N}_t$ and $X_{t-}:=\lim_{s<t,\, s\rightarrow t}X_s$.
\end{lemma}
\begin{proof}
By It\^{o}'s formula (see e.g. \cite[Corollary 3.6]{GaYor2006}), we have
\begin{eqnarray*}
\mathcal{N}_t&=&\int_0^t\langle\nabla H_\kappa(T-s)f(X_s),\d B_s\rangle +\\ &&\sum_{\alpha\in\mathfrak{R}_+}\int_0^t\sqrt{\kappa_\alpha}\,\frac{H_\kappa(T-s)f(X_{s-}) - H_\kappa(T-s)f(r_\alpha X_{s-})}{\langle\alpha, X_{s-}\rangle}\,\d M_s^\alpha,
\end{eqnarray*}
where $(M_t^\alpha)_{t\geq0}$ is an one-dimensional martingale with discontinuous paths. Hence $(\mathcal{N}_t)_{t\in[0,T]}$ is a martingale. From \cite[Theorem 1]{GaYor2006}, we have $\langle M^\alpha\rangle_t=2t$. Thus, we immediately get \eqref{quad-var-N}.
\end{proof}

For $f,g\in C^4(\R^d)$, we define as in the classical Laplacian case (see e.g. \cite{BE1985}) that
\begin{eqnarray}\label{Gamma2}
	\Gamma_2(f,g):=\frac{1}{2}\big[\Delta_\kappa \Gamma(f,g)-\Gamma(\Delta_\kappa f,g)-\Gamma(f,\Delta_\kappa g)\big].
\end{eqnarray}
For convenience, we set $\Gamma_2(f)=\Gamma_2(f,f)$.

The following result is the key to apply the approach in \cite{Li2019} mentioned above, and it may be of independent interest. Since the proof is a little bit long, we present the details in the next section.
\begin{proposition}\label{cur-dim}
Let $G$ be isomorphic to $\mathbb{Z}_2^d$. Then, for every $f\in C^4(\R^d)$,
$$\Gamma_2(f)\geq \|\Hess(f)\|_{\rm HS}^2.$$
where $\Hess(f)$ is the Hessian of $f$ and $\|\cdot\|_{\rm HS}$ is the Hilbert--Schmidt norm.
\end{proposition}

\begin{remark}
It should be regarded as that, in the particular setting of Proposition \ref{cur-dim}, the Dunkl Laplacian $\Delta_\kappa$ satisfies the curvature-dimension condition CD$(0,\infty)$ in the sense of Bakry--Emery \cite{BE1985}.
\end{remark}

Applying Proposition \ref{cur-dim}, we immediately derive \eqref{BL-grad} below, which may be regarded as the gradient estimate for Dunkl heat flows  in the sense of Bakry--Ledoux (see e.g. \cite{BL2006} for the diffusion heat flow case). The proof follows from the standard heat flow interpolation approach.
\begin{corollary}\label{grad-est}
 The assertion that
\begin{eqnarray}\label{BL-grad}
\Gamma\big(H_\kappa(t)f\big)\leq H_\kappa(t)\Gamma(f),\quad f\in C^{\infty}_c(\R^d),\,t>0,
\end{eqnarray}
is equivalent to
$$\Gamma_2(f)\geq0,\quad f\in C^\infty(\R^d).$$
\end{corollary}
\begin{proof}  Let $0\leq s\leq t$. Then, for every $f\in C^{\infty}_c(\R^d)$,
\begin{eqnarray*}
H_\kappa(t)\Gamma(f)-  \Gamma\big(H_\kappa(t)f\big) &=&\int_0^t\frac{\d}{\d s}H_\kappa(s)\Gamma\big(H_\kappa(t-s)f\big)\,\d s\\
&=&\int_0^t H_\kappa(s)\big[\Delta_\kappa\Gamma\big(H_\kappa(t-s)f\big)  \\
&&-  2\Gamma\big(\Delta_\kappa H_\kappa(t-s)f,H_\kappa(t-s)f\big)\big]\,\d s\\
&=&2 \int_0^t H_\kappa(s)\Gamma_2\big(H_\kappa(t-s)f\big)\,\d s,
\end{eqnarray*}
where the second equality can be checked without much effort. Thus, the equivalence of both assertions is clear.
\end{proof}

Define another square function $\tilde{g}(f)$ as
\begin{eqnarray*}
\tilde{g}(f)(x)=\Big(\int_0^\infty H_\kappa(t)\Gamma\big(H_\kappa(t)f\big)\,\d t\Big)^{1/2},\quad x\in\R^d.
\end{eqnarray*}
Then, for every $x\in\R^d$, we have
\begin{eqnarray}\label{square-function-compare}
g_\Gamma(f)^2(x)&=&\int_0^\infty \Gamma\big(H_\kappa(t)f\big)(x)\,\d t\cr
&=&\int_0^\infty \Gamma\big(H_\kappa(t/2)H_\kappa(t/2)f\big)(x)\,\d t\cr
&\leq&\int_0^\infty H_\kappa(t/2)\Gamma\big(H_\kappa(t/2)f\big)(x)\,\d t\cr
&=&2\tilde{g}(f)^2(x),
\end{eqnarray}
where we used Corollary \ref{grad-est} in the above inequality.

Let
\begin{eqnarray*}
\tilde{g}_T(f)(x):=\Big(\int_0^T H_\kappa(t)\Gamma\big(H_\kappa(t)f\big)\,\d t\Big)^{1/2},\quad x\in\R^d,
\end{eqnarray*}
Then, it is immediately to see that, for every $x\in\R^d$, $\tilde{g}_T(f)(x)$ increases to $\tilde{g}(f)(x)$ as $T$ goes to $\infty$. The key point here is that $\tilde{g}_T(f)(x)$ can be expressed as an integral of the conditional expectation of the predictable quadratic variation $\langle\mathcal{N}\rangle_T$ as the next lemma shows. For instance, see \cite{BBL2016} for the case of  pure jump L\'{e}vy processes and \cite{Li2017} for the case of diffusion processes.
\begin{lemma}\label{prob-repr}
Let $T>0$. For every $f\in C_c^\infty(\R^d)$ and every $x\in\R^d$,
\begin{eqnarray}\label{prob-reprez}
\tilde{g}_T(f)(x)=\Big(\frac{1}{2}\int_{\R^d} \E_y\big(\langle\mathcal{N}\rangle_T \big| X_T=x\big)h_\kappa(T,x,y)\,\d\mu_\kappa(y)\Big)^{1/2},
\end{eqnarray}
where $\E_y$ denotes the expectation of the process $(X_t)_{t\geq0}$ starting from $y$.
\end{lemma}
\begin{proof}
Indeed, by the change-of-variables formula, the stochastic completeness \eqref{stoch-complete} and \eqref{quad-var-N}, we have
\begin{eqnarray*}
&&\tilde{g}_T(f)^2(x)\\
&=&\int_0^T\!\!\!\int_{\R^d}h_\kappa(t,x,z)\Gamma\big(H_\kappa(t)f\big)(z)\,\d\mu_\kappa(z)\d t\\
&=&\int_0^T\!\!\!\int_{\R^d} h_\kappa(T-t,z,x)\Gamma\big(H_\kappa(T-t)f\big)(z)\,\d\mu_\kappa(z)\d t\\
&=&\int_0^T\!\!\!\int_{\R^d} h_\kappa(T-t,z,x)\Gamma\big(H_\kappa(T-t)f\big)(z)\Big(\int_{\R^d}h_\kappa(t,z,y)\,\d\mu_\kappa(y)\Big)
\,\d\mu_\kappa(z)\d t\\
&=&\int_{\R^d}\!\!\Big(\int_0^T\!\!\!\int_{\R^d}\frac{h_\kappa(t,y,z)h_\kappa(T-t,z,x)}{h_\kappa(T,y,x)}\Gamma\big(H_\kappa(T-t)f\big)(z)\,\d\mu_\kappa(z)\d t\Big)h_\kappa(T,y,x)\d\mu_\kappa(y)\\
&=&\int_{\R^d}\E_y\bigg[\int_0^T\Big(\sum_{\alpha\in\mathfrak{R}_+}\kappa_\alpha\,\frac{\big(H_\kappa(T-t)f(X_{t-})-H_\kappa(T-t)f(r_\alpha X_{t-})\big)^2}{\langle\alpha,X_{t-}\rangle^2}\\
&& + |\nabla H_\kappa(T-t)f(X_t)|^2 \Big)\,\d t\Big|X_T=x\bigg]h_\kappa(T,x,y)\,\d\mu_\kappa(y)\\
&=&\frac{1}{2}\int_{\R^d} \E_y\big(\langle\mathcal{N}\rangle_T\big|X_T=x\big)h_\kappa(T,x,y)\,\d\mu_\kappa(y).
\end{eqnarray*}\end{proof}

Now we are ready to present the main result in this section.
\begin{theorem}\label{Gamma-LPS-big-p}
Let $p\in[2,\infty)$. Suppose the Weyl group $G$ is isomorphic to $\mathbb{Z}_2^d$. Then the operator $g_\Gamma$ is bounded in $L^p(\mu_\kappa)$; more precisely, there exists a positive constant $C(p)$ such that
$$\|g_\Gamma(f)\|_{L^p(\mu_\kappa)}\leq C(p)\|f\|_{L^p(\mu_\kappa)},\quad f\in L^p(\mu_\kappa);$$
moreover, $C(p)$ depends only on $p$.
\end{theorem}
The proof is the same as \cite[Theorem 4.4]{Li2019} by  combining \eqref{square-function-compare} and \eqref{prob-reprez} together, applying the Burkholder--Davis--Gundy inequality and the monotone convergence theorem, and by standard approximation. We omit the details here.

Finally, we give a remark on Theorem \ref{Gamma-LPS-big-p}, pointed out by an anonymous referee.
\begin{remark}
The analytic proof in \cite[pages 52-55]{Stein70} can be possibly adapted to prove that $g_\Gamma$ is bounded in $L^p(\mu_\kappa)$ for all $2<p<\infty$ when $G$ is isomorphic to $\mathbb{Z}_2^d$.
\end{remark}

\section{Proofs of Proposition \ref{cur-dim}}\hskip\parindent
In this section, we present the details on the proof of Proposition \ref{cur-dim}. As an illustration, we first consider the rank-one case following the notations of Example \ref{rank-one} introduced in Section 1.

\begin{corollary}\label{R-Gamma2}
In the case of Example \ref{rank-one} with $\kappa\in\R$, for every $f\in C^4(\R)$,
\begin{eqnarray}\label{R-Gamma2-0}
\Gamma_2(f)(x)&=&f''(x)^2+\frac{\kappa}{x^2}\Big[f'(x)+f'(-x)-\frac{f(x)-f(-x)}{x}\Big]^2\cr
&&+\frac{\kappa}{2x^2}\Big[2f'(x)-\frac{f(x)-f(-x)}{x}\Big]^2,\quad x\in\R.
\end{eqnarray}
\end{corollary}
\begin{proof}
By \eqref{Gamma2}, we see
\begin{eqnarray}\label{R-Gamma2-1}
\Gamma_2(f)(x)=\frac{1}{2}\Delta_\kappa\Gamma(f)(x)-\Gamma(\Delta_\kappa f, f)(x).
\end{eqnarray}

We calculate the two terms on the right hand side of \eqref{R-Gamma2-1}. By \eqref{Gamma}, direct calculation leads to
\begin{eqnarray}\label{R-Gamma2-2}
\Delta_\kappa\Gamma(f)(x)&=&\Gamma(f)''(x)+\frac{\kappa}{x^2}\big[2x\Gamma(f)'(x)+\Gamma(f)(-x)-\Gamma(f)(x)\big]\cr
&=&2f''(x)^2+2f'(x)f'''(x)+\frac{4\kappa}{x}f'(x)f''(x)\cr
&&+\frac{\kappa}{x^2}\Big\{\big[f'(-x)^2-f'(x)^2\big]
+\big[f(x)-f(-x)\big]\big[f''(x)-f''(-x)\big] +\big[f'(-x)+f'(x)\big]^2\Big\}\cr
&&+2\frac{\kappa^2-2\kappa}{x^3}\big[f(x)-f(-x)\big]\big[f'(x)+f'(-x)\big]\cr
&&+\frac{3\kappa-2\kappa^2}{x^4}\big[f(x)-f(-x)\big]^2.
\end{eqnarray}

Similarly, by \eqref{Gamma}, \eqref{Delta} and direct calculation, we have
\begin{eqnarray}\label{R-Gamma2-3}
\Gamma(\Delta_\kappa f, f)(x) &=& (\Delta_\kappa f)'(x)f'(x)+\frac{\kappa}{2x^2}\big[f(x)-f(-x)\big]\big[(\Delta_\kappa f)(x)-(\Delta_\kappa f)(-x)\big]\cr
&=&f'(x)f'''(x)+\frac{2\kappa}{x} f'(x)f''(x)\cr
&&+\frac{\kappa}{2x^2}\Big\{-6f'(x)^2-2f'(x)f'(-x)
+\big[f(x)-f(-x)\big]\big[f''(x)-f''(-x)\big]\Big\}\cr
&&+ \frac{1}{x^3}\Big\{\kappa^2\big[f(x)-f(-x)\big]\big[f'(x)+f'(-x)\big]+ 2\kappa f'(x)\big[f(x)-f(-x)\big]\Big\}\cr &&-\frac{\kappa^2}{x^4}\big[f(x)-f(-x)\big]^2.
\end{eqnarray}

Thus, combining \eqref{R-Gamma2-1}, \eqref{R-Gamma2-2} and \eqref{R-Gamma2-3}, we obtain
\begin{eqnarray*}
\Gamma_2(f)(x)&=&f''(x)^2+ \kappa\bigg\{ \frac{\big[f'(x)+f'(-x)\big]^2}{x^2}+\frac{\big[f(x)-f(-x)\big]^2}{x^4}\\
&&-2\frac{\big[f(x)-f(-x)\big]\big[f'(x)+f'(-x)\big]}{x^3}\bigg\}+\frac{\kappa}{2}\bigg\{\frac{\big[f(x)-f(-x)\big]^2}{x^4}\\
&& +4\frac{f'(x)^2}{x^2} -4\frac{f'(x)\big[f(x)-f(-x)\big]}{x^3}\bigg\}\\
&=&f''(x)^2+\frac{\kappa}{x^2}\Big[f'(x)+f'(-x)-\frac{f(x)-f(-x)}{x}\Big]^2\\
&&+\frac{\kappa}{2x^2}\Big[2f'(x)-\frac{f(x)-f(-x)}{x}\Big]^2.
\end{eqnarray*}
We finish the proof.
\end{proof}

Now we are ready to present the proof on Proposition \ref{Gamma2} which is much more complicated than the proof of Corollary \ref{R-Gamma2}. Assume that the Weyl group $G$ is isomorphic to $\mathbb{Z}_2^d$. The crucial fact is that, in this case, we have
\begin{eqnarray}\label{ortho}
    	\langle\alpha,\beta\rangle=0,\quad\mbox{for every }\alpha,\beta\in\mathfrak{R}_+\mbox{ with }\alpha\neq\beta.
    \end{eqnarray}
\begin{proof}[Proof of Proposition \ref{cur-dim}]
Let $x\in\R^d$, $\alpha=(\alpha_1,\cdots,\alpha_d)\in\mathfrak{R}_+$ and $y=(y_1,\cdots,y_d)\in\R^d$. We need the following elementary facts.
\begin{eqnarray*}
&&\Delta|\nabla f(x)|^2-2\langle \nabla \Delta f(x),\nabla f(x)\rangle=2\|\Hess(f)(x)\|_{{\rm HS}}^2;\\
&&\nabla \frac{1}{\langle\alpha,\cdot\rangle(x)}=-\frac{\alpha}{\langle\alpha,x\rangle^2},\quad
\nabla \frac{1}{\langle\alpha,\cdot\rangle^2(x)}=-\frac{2\alpha}{\langle\alpha,x\rangle^3},\quad\langle\alpha,x\rangle\neq0;\\
&&\Delta \frac{1}{\langle\alpha,\cdot\rangle^2(x)}=\frac{12}{\langle\alpha,x\rangle^4},\quad \langle\alpha,x\rangle\neq0;\\
&&\nabla[f(\cdot)-f(r_{\alpha}\cdot)](x)=\nabla f(x)-\nabla f(r_{\alpha}x)+\langle \nabla f(r_{\alpha}x),\alpha\rangle\alpha;\\
&&\Delta[f(\cdot)-f(r_{\alpha}\cdot)](x)=\Delta f(x)-\Delta f(r_{\alpha}x);\\
&&\big\langle \nabla|\nabla f(\cdot)|^2(x),y\big\rangle=2\sum_{ i,j=1}^d \partial_{x_i}f(x)\partial_{x_ix_j}^2f(x)y_j;\\
&&\big\langle\nabla\langle\nabla f(\cdot),\alpha\rangle(x), y\big\rangle=\sum_{ i,j=1}^d \alpha_i\partial^2_{x_ix_j}f(x)y_j.
\end{eqnarray*}

Recall that $|\alpha|=\sqrt{2}$ and $\delta_\alpha f(x):=[f(x)-f(r_\alpha x)]/\langle \alpha,x\rangle$ (see Remark \ref{remark}(i)).
With the above facts, by \eqref{Gamma2}, \eqref{Gamma} and \eqref{Delta}, we do the similar calculation as the proof of Corollary \ref{R-Gamma2}. However, due to the mixture of local and non-local terms in \eqref{Delta} and \eqref{Gamma}, the non-local term in the expression of $\Gamma_2(f)$ consists of two summations: one is over $\alpha\in\mathfrak{R}_+$ and the other is over $\alpha,\beta\in\mathfrak{R}_+$. Indeed,
\begin{eqnarray}\label{Gamma2-1}
	\Gamma_2(f)(x)&=&\|\Hess(f)(x)\|_{\rm HS}^2\cr
	&+&\sum_{ \alpha\in \mathfrak{R}_+} \frac{2\kappa_{\alpha}}{\langle \alpha,x\rangle^2 }\left\{\left| \delta_{\alpha}f(x)\alpha-\nabla[f(x)-f(r_{\alpha}x) ]  \right|^2  +\left| \langle \nabla f(x),\alpha\rangle - \delta_{\alpha}f(x) \right|^2
	\right\}\cr
	&+&\sum_{\alpha,\beta\in \mathfrak{R}_+} \frac{\kappa_{\alpha}\kappa_{\beta} }{\langle \beta,x\rangle^2 }
	\bigg\{   -\frac{2[\delta_{\alpha}f(x)]^2}{\langle\alpha,x\rangle}\langle\alpha,\beta\rangle \langle \beta,x\rangle  -[\delta_{\alpha}f(x)]^2  +[\delta_{\alpha}f(r_{\beta}x)]^2\cr
&+&\frac{2 \delta_{\alpha}f(x)}{\langle\alpha,x\rangle}\langle\nabla [f(x)-f(r_{\alpha}x)],\beta\rangle\langle \beta,x\rangle \cr
	&-&2\Big[\Big( \frac{\langle \nabla f(x),\alpha\rangle }{\langle\alpha,x\rangle}-\frac{\delta_{\alpha}f(x) }{\langle\alpha,x\rangle} \Big)- \Big( \frac{\langle \nabla f(r_{\beta}x),\alpha\rangle }{\langle\alpha,r_{\beta}x\rangle}-\frac{ \delta_{\alpha} f(r_{\beta}x)  }{\langle\alpha,r_{\beta}x\rangle} \Big)    \Big] [f(x)-f(r_{\beta}x)]\bigg\}\cr
	&=:&\|\Hess(f)(x)\|_{\rm HS}^2+A+B,
\end{eqnarray}
where $A$ denotes the summation of terms over $\alpha\in\mathfrak{R}_+$, and $B$ denotes the summation of terms over $\alpha,\beta\in\mathfrak{R}_+$.
 (In the rank-one case, $B=0$ and $A$ reduces to the sum of the last two terms in the right hand side of \eqref{R-Gamma2-0}.)

For $B$, we further split it into two summations: one is over $\alpha,\beta\in\mathfrak{R}_+$ with $\alpha=\beta$, and the other is over $\alpha,\beta\in\mathfrak{R}_+$ with $\alpha\neq\beta$. Using the facts that $|\alpha|=\sqrt{2}$, $\langle \alpha ,r_{\alpha}x\rangle =-\langle \alpha,x\rangle$, and $r_{\alpha} r_{\alpha}x=x$ for all $\alpha\in\mathfrak{R}_+$,  it is easy to see that the summation over $\alpha,\beta\in\mathfrak{R}_+$ with $\alpha=\beta$ equals $0$. For any $\alpha,\beta\in\mathfrak{R}_+$ with $\alpha\neq\beta$, by \eqref{ortho}, we have $\langle \alpha, r_{\beta}x\rangle=\langle \alpha,x\rangle$. Thus,
    \begin{eqnarray}\label{B}
        B&=&\sum_{\alpha,\beta\in \mathfrak{R}_+,\,\alpha\neq \beta} \frac{\kappa_{\alpha}\kappa_{\beta} }{\langle \beta,x\rangle^2 }
        \bigg\{   \frac{2 \delta_{\alpha}f(x)}{\langle\alpha,x\rangle}\langle\nabla [f(x)-f(r_{\alpha}x)],\beta\rangle\langle \beta,x\rangle -[\delta_{\alpha}f(x)]^2  +[\delta_{\alpha}f(r_{\beta}x)]^2  \cr
        && \qquad -2\Big[\Big( \frac{\langle \nabla f(x),\alpha\rangle }{\langle\alpha,x\rangle}-\frac{\delta_{\alpha}f(x) }{\langle\alpha,x\rangle} \Big)- \Big( \frac{\langle \nabla f(r_{\beta}x),\alpha\rangle }{\langle\alpha, x\rangle}-\frac{ \delta_{\alpha} f(r_{\beta}x)  }{\langle\alpha, x\rangle} \Big)    \Big] [f(x)-f(r_{\beta}x)]\bigg\}.
    \end{eqnarray}

Further split the right hand side of \eqref{B} into two parts: one is the summation of terms involving the gradient $\nabla$, denoted by $B_1$, and the other is the summation of the remaining terms in the right hand of \eqref{B}, denoted by $B_2$.

    Consider $B_1$. For any $\alpha,\beta\in \mathfrak{R}_+$ with $\alpha\neq \beta$, by \eqref{ortho} again, it is clear that $\langle \nabla [f(r_{\beta}x)],\alpha\rangle=\langle \nabla f(r_{\beta}x),\alpha\rangle$. Then
    \begin{eqnarray}\label{B1}
    	B_1
    	&=&\sum_{\alpha,\beta\in \mathfrak{R}_+,\,\alpha\neq \beta} \frac{\kappa_{\alpha}\kappa_{\beta} }{\langle \beta,x\rangle^2 }
    	\bigg\{   \frac{2 \delta_{\alpha}f(x)}{\langle\alpha,x\rangle}\langle\nabla [f(x)-f(r_{\alpha}x)],\beta\rangle\langle \beta,x\rangle  \cr
    	&&\qquad  -2\Big[  \frac{\langle \nabla f(x),\alpha\rangle }{\langle\alpha,x\rangle} - \frac{\langle \nabla f(r_{\beta}x),\alpha\rangle }{\langle\alpha, x\rangle}   \Big] [f(x)-f(r_{\beta}x)]\bigg\}\cr
    	&=&\sum_{\alpha,\beta\in \mathfrak{R}_+,\,\alpha\neq \beta}  \frac{4\kappa_{\alpha}\kappa_{\beta}}{\langle \alpha,x\rangle \langle \beta,x\rangle}
    	\Big\{    \delta_{\alpha}f(x)  \langle\nabla [f(x)-f(r_{\alpha}x)],\beta\rangle  -  \delta_{\beta}f(x)   \langle \nabla [f(x)-f(r_{\beta}x)], \alpha\rangle \Big\}\cr
    	&=&0,
    \end{eqnarray}
where we used the symmetry of $\alpha$ and $\beta$ in the last equality.

    Consider $B_2$. For any $\alpha,\beta\in \mathfrak{R}_+$ with $\alpha\neq \beta$, again by \eqref{ortho}, $r_{\alpha}r_{\beta}x=r_{\beta}r_{\alpha} x$ and $\langle \alpha, r_{\beta}x\rangle=\langle \alpha,x\rangle$. Then, we obtain
    \begin{eqnarray}\label{B2}
    	B_2
    	&=&\sum_{\alpha,\beta\in \mathfrak{R}_+,\,\alpha\neq \beta} \frac{\kappa_{\alpha}\kappa_{\beta} }{\langle \beta,x\rangle^2 }
    	\bigg\{ [\delta_{\alpha}f(r_{\beta}x)]^2  -[\delta_{\alpha}f(x)]^2   -2\Big[-\frac{\delta_{\alpha}f(x) }{\langle\alpha,x\rangle} +\frac{ \delta_{\alpha} f(r_{\beta}x)  }{\langle\alpha, x\rangle}     \Big] [f(x)-f(r_{\beta}x)]\bigg\}\cr
    	&=&\sum_{\alpha,\beta\in \mathfrak{R}_+,\,\alpha\neq \beta} \frac{\kappa_{\alpha}\kappa_{\beta} }{\langle \alpha,x\rangle^2 \langle \beta,x\rangle^2 }
    	\Big\{   -[f(x)-f(r_{\alpha}x) ]^2  +[f(r_{\beta}x)-f(r_{\alpha}r_{\beta}x)]^2 \cr
    	&&\qquad -2\big[-\big(f(x)-f(r_{\alpha}x) \big) + \big(f(r_{\beta}x)-f(r_{\alpha}r_{\beta}x)\big)  \big] [f(x)-f(r_{\beta}x)]\Big\}\cr
    	&=&\sum_{\alpha,\beta\in \mathfrak{R}_+,\,\alpha\neq \beta} \frac{\kappa_{\alpha} \kappa_{\beta} }{ \langle \alpha,x\rangle^2\langle \beta,x\rangle^2 }
    \big[ f(x)-f(r_{\alpha}x)-f(r_{\beta}x)+f(r_{\alpha} r_{\beta}x)  \big]\cr
    &&\qquad \times\big[ f(x)+f(r_{\alpha}x)-3f(r_{\beta}x)+f(r_{\alpha} r_{\beta}x)  \big]\cr
    	&=&\sum_{ \alpha,\beta\in \mathfrak{R}_+,\,\alpha\neq \beta} \frac{\kappa_{\alpha} \kappa_{\beta} }{ \langle \alpha,x\rangle^2\langle \beta,x\rangle^2 } \big[ f(x)-f(r_{\alpha}x)-f(r_{\beta}x)+f(r_{\alpha} r_{\beta}x)   \big]^2,
    \end{eqnarray}
where, in the last equality, we first used the symmetry of $\alpha$ and $\beta$ and then took the average.

  Therefore, combining \eqref{Gamma2-1}, \eqref{B}, \eqref{B1} and \eqref{B2}, we finally arrive at
  \begin{eqnarray*}
	\Gamma_2(f)(x)&=&\|\Hess(f)(x)\|_{\rm HS}^2\cr
	&&\quad+\sum_{ \alpha\in \mathfrak{R}_+} \frac{2\kappa_{\alpha}}{\langle \alpha,x\rangle^2 }\left(\left| \delta_{\alpha}f(x)\alpha-\nabla[f(x)-f(r_{\alpha}x) ]  \right|^2  +\left| \langle \nabla f(x),\alpha\rangle - \delta_{\alpha}f(x) \right|^2
	\right)\cr
&&\quad+\sum_{\alpha,\beta\in \mathfrak{R}_+,\,\alpha\neq \beta} \frac{\kappa_{\alpha} \kappa_{\beta} }{ \langle \alpha,x\rangle^2\langle \beta,x\rangle^2 } \big[ f(x)-f(r_{\alpha}x)-f(r_{\beta}x)+f(r_{\alpha} r_{\beta}x)   \big]^2\cr
&\geq& \|\Hess(f)(x)\|_{\rm HS}^2,	
\end{eqnarray*}
which completes the proof of Proposition \ref{cur-dim}.
\end{proof}

\begin{remark} In the above proof, assumption \eqref{ortho} on the Weyl group $G$ is employed only to show that ${B}\geq0$. With \eqref{ortho}, the terms like
$$\langle\alpha,\beta\rangle\frac{\langle \alpha,x\rangle}{\langle \beta,x\rangle}|f(x)-f(r_{\alpha}x)|^2   ,\quad f(r_\alpha x)\big[f(r_{\beta} r_{\alpha}x)-   f(r_{\alpha}r_{\beta}x)\big]$$
clearly disappear. However, without \eqref{ortho}, it seems difficult to deal with the above terms and we do not know the sign of $B$.
It seems interesting to find or construct an example such that $B_2$ is negative, which essentially results in the negativity of $\Gamma_2$.
\end{remark}

\subsection*{Acknowledgment}\hskip\parindent
The authors much appreciate the anonymous referees for their careful corrections, helpful suggestions and comments, which improve the manuscript a lot. The authors also thank Prof.  B\'{e}chir Amri for kindly sending them a copy of reference \cite{AmSi2012}.  The first named author would like to thank Prof. Hong-Quan Li (from Fudan University) for introducing the Dunkl operator to him, and to acknowledge the financial supports from the National Natural Science Foundation of China (Nos. 11571347 and 11831014). The second named author is supported by the National Natural Science Foundation of China (No.11801404).


\begin{thebibliography}{a23}
\bibitem{AmSi2012}
B. Amri, M. Sifi: Singular integral operators in Dunkl setting. J. Lie Theory 22 (2012), 723--739.

\bibitem{Anker2017}
J.-P. Anker: An introduction to Dunkl theory and its analytic aspects. In \emph{Analytic, algebraic and geometric aspects of differential equations}. Trends in Math., Birkh\"{a}user, Cham, 2017, pp. 3--58.

\bibitem{ADH2019}
J.-P. Anker, J. Dziuba\'{n}ski, A. Hejna: Harmonic Functions, Conjugate Harmonic Functions and the Hardy Space $H^1$ in the Rational Dunkl Setting. J. Fourier Anal. Appl. 25 (2019), 2356--2418.

\bibitem{DH2020}
J. Dziuba\'{n}ski, A. Hejna: Upper and lower bounds for Littlewood-Paley square functions in the Dunkl setting. Preprint (2020), arXiv:2005.00793v2.

\bibitem{BE1985} D. Bakry, M. Emery: Diffusions hypercontractives. In \emph{S\'{e}minaire de probabilit\'{e}s XIX, 1983/84}. Lecture Notes in Math. 1123, Springer, Berlin, 1985, pp.177--206.

\bibitem{BL2006}
D. Bakry, M. Ledoux: A logarithmic Sobolev form of the Li-Yau parabolic inequality. Rev. Mat. Iberoam. 22(2) (2006), 683--702.

\bibitem{BBL2016}
R. Ba\~{n}uelos, K. Bogdan, T. Luks: Hardy--Stein identities and square functions for semigroups. J. London Math. Soc. 94 (2016), 462--478.

\bibitem{BMH2003}
R. Ba\~{n}uelos, P.J. M\'{e}ndez-Hern\'{a}ndez: Space-time Brownian motion and the Beurling--Ahlfors transform. Indiana Univ.
Math. J. 52 (2003), 981--990.

\bibitem{DX2015}
F. Dai, Y. Xu: \emph{Analysis on $h$-harmonics and Dunkl transforms}. In: Advanced Courses in Mathematics, CRM Barcelona. Edited by Sergey Tikhonov.  Birkh\"{a}user/Springer, Basel, 2015.

\bibitem{Del2011}
L. Deleaval: Two results on the Dunkl maximal operator. Studia Math. 203 (1) (2011), 47--68.

\bibitem{Nick08}
N. Dungey: A Littlewood--Paley--Stein estimates on graphs and groups. Studia Math. 189 (2008), 113--129.

\bibitem{Dunkl1988}
C.F. Dunkl: Reflection groups and orthogonal polynomials on the sphere. Math. Z. 197 (1988), 33--60.

\bibitem{Dunkl1989}
C.F. Dunkl: Differential-difference operators associated to reflection groups. Trans. Amer. Math. Soc. 311 (1989), 167--183.

\bibitem{DunklXu2014}
C.F. Dunkl, Y. Xu: \emph{Orthogonal polynomials of several variables}. Encyclopedia of Mathematics and its Applications 155, Cambridge University Press, Cambridge, Second edition, 2014.

\bibitem{GaYor2005}
L. Gallardo, M. Yor: Some new examples of Markov processes which enjoy the time-inversion property. Probab. Theory Relat. Fields 132 (2005), 150--162.

\bibitem{GaYor2006}
L. Gallardo, M. Yor: A chaotic representation property of the multidimensional Dunkl processes. Ann. Probab. 34 no. 4 (2006), 1530--1549.

\bibitem{GS2020} P. Graczyk, P. Sawyer: Sharp Estimates of Radial Dunkl and Heat Kernels in the Complex Case $A_n$. Preprint (2020), arXiv:2012.12022.

\bibitem{GLR2018}
P. Graczyk, T. Luks, Margit R\"{o}sler: On the Green function and Poisson integrals of the Dunkl Laplacian. Potential Anal. 48 no. 3 (2018),  337--360.

\bibitem{Li2017}
H. Li: Weighted Littlewood--Paley inequalities for heat flows in $\RCD$ spaces. J. Math. Anal. Appl. 479 (2019), 1618--1640.

\bibitem{Li2019}
H. Li: Littlewood--Paley--Stein inequalities on $\RCD(K,\infty)$ spaces. Preprint (2019), arXiv:1905.01432.

\bibitem{Li2020}
H. Li: Weak Type Estimates for Square Functions of Dunkl Heat Flows. Preprint (2021), arXiv:2101.04056.

\bibitem{LiWang16}
H. Li, J. Wang: Littlewood--Paley--Stein estimates for non-local Dirichlet forms. Preprint (2017), arXiv:1704.02690v3. To appear in J. d'Analyse Math., doi 10.1007/s11854-021-0148-5.

\bibitem{LiWang2019}
H. Li, J. Wang: Littlewood--Paley--Stein functions for non-local Schr\"{o}dinger operators. Positivity 24 (2020), 1293--1312.

\bibitem{LZL2017}
J. Liao, X. Zhang, Z. Li: On Littlewood--Paley functions associated with the Dunkl operator. Bull. Aust. Math. Soc. 96 (2017), 126--138.

\bibitem{Rosler1998}
M. R\"{o}sler: Generalized Hermite polynomials and the heat equation for Dunkl operators. Comm. Math. Phys. 192 no. 3 (1998), 519--542.

\bibitem{Rosler2003}
M. R\"{o}sler: Dunkl operators: Theory and Applications. In \emph{Orthogonal Polynomials and Special Functions}, Leuven 2002, ed, by E. Koelink, W. Van Assche. Lecture Notes in Mathematics, vol. 1817, Springer, Berlin, 2003, pp. 93--135.

\bibitem{RosVoi1998}
M. R\"{o}sler, M. Voit: Markov processes related with Dunkl operators. Adv. App. Math. 21 (1998), 575--643.

\bibitem{RosVoi2008} M. R\"{o}sler, M. Voit: Dunkl theory, convolution algebras, and related Markov processes. In \emph{Harmonic
and stochastic analysis of Dunkl processes}, P. Graczyk, M. R\"{o}sler, M. Yor (eds.), 1--112, Travaux en cours 71, Hermann, Paris, 2008.

\bibitem{Soltani2005}
F. Soltani: Littlewood--Paley operators associated with the Dunkl operator on $\R$. J. Funct. Anal. 221 (2005), 205--225.

\bibitem{FSoltani2005}
F. Soltani: Littlewood-Paley $g$-function in the Dunkl analysis on $\R^d$. J. Ineq. Pure and Appl. Math. 6 (3) (2005).

\bibitem{Stein1958}
E.M. Stein: On the functions of Littlewood-Paley, Lusin, and Marcinkiewicz. Trans. Amer. Math. Soc. 88 (1958), 430--466.

\bibitem{St1970}
E.M. Stein: \emph{Singular Integrals and Differentiability Properties of Functions}. Princeton Mathematical Series, No. 30, Princeton
University Press, Princeton, 1970.

\bibitem{Stein70} E.M. Stein: \emph{Topics in Harmonic Analysis Related to the Littlewood--Paley Theory}. Ann. of Math. Stud. 63, Princeton
    Univ. Press, Princeton, 1970.

 \bibitem{Velicu2020}
A. Velicu: Sobolev-type inequalities for Dunkl operators. J. Funct. Anal. 279 no. 7 (2020), 108695, 37pp.

\bibitem{YJA}
J.A. Yan: \emph{Lecture on Measure Theory}. Second edition, Science Press, Bejing, 2004.
\end{thebibliography}
\end{document}